\UseRawInputEncoding
\documentclass[12pt]{amsart}

\usepackage[all]{xy}
\usepackage{fullpage}
\usepackage{latexsym}
\usepackage{amsmath}
\usepackage{amsfonts}
\usepackage{amssymb}
\usepackage{amsthm}
\usepackage{eucal}
\usepackage{enumerate,yfonts}
\usepackage{mathrsfs}
\usepackage{graphicx}
\usepackage{graphics}
\usepackage{epstopdf}
\usepackage{amscd}
\usepackage{bbm}
\usepackage{hyperref}
\usepackage{url}
\usepackage{color}
\usepackage{bbm}
\usepackage{cancel}
\usepackage{enumerate}
\usepackage{amsmath,amsthm}
\usepackage{amssymb}
\usepackage{epsfig}
\usepackage{pstricks}
\usepackage{xy}
\usepackage{xypic}

\newtheorem{thm}{Theorem}[section]
\newtheorem{corollary}[thm]{Corollary}
\newtheorem{lemma}[thm]{Lemma}

\newtheorem{proposition}[thm]{Proposition}

\newtheorem{thm-dfn}[thm]{Theorem-Definition}

\theoremstyle{definition}
\newtheorem{definition}[thm]{Definition}
\newtheorem{remark}[thm]{Remark}
\newtheorem{example}[thm]{Example}

\numberwithin{equation}{section}

\newcommand{\quash}[1]{}  

\newcommand{\rW}{{\mathrm W}}

\newcommand{\bG}{{\mathbb G}}

\newcommand{\mF}{\mathcal{F}}

\newcommand{\mL}{\mathcal{L}}

\newcommand{\on}{\operatorname}

\newcommand{\is}{\simeq}
\newcommand{\ra}{\rightarrow}

\newcommand{\calC}{\mathcal C}
\newcommand{\barQ}{\bar{\mathbb Q}_\ell}

\begin{document}

\title{Functorial transfer for reductive groups and central complexes}

\begin{abstract}
A class of Weyl group equivariant 
$\ell$-adic complexes on a torus, called the 
central complexes, was introduced and studied in our previous work on Braverman-Kazhdan conjecture. In this note 
we show that the category of central complexes 
admits functorial monoidal  transfers with respect to morphisms between the dual groups.
Combining with the work of Bezrukavnikov-Deshpande, we show that the $\ell$-adic bi-Whittaker categories (resp. the category of vanishing complexes, the category of 
stable complexes) on reductive groups 
admit functorial transfers.

As an application, we give a new geometric proof of Laumon-Letellier's fromula for  transfer maps of stable functions on finite reductive groups.

\quash{
Let $G$ and $G'$ be two connected reductive groups
over an algebraically closed field.
Given a morphism 
   $\hat\rho:\hat G'\to \hat G$ between their Langlands dual groups, we construct a functorial monoidal  transfer morphism between the $\ell$-adic bi-Whittaker category 
   (resp. the category 
   vanishing complexes)
   of $G$ 
 and $G'$.
   The construction 
relies on the  work of Bezrukavnikov-Deshpande 
   on  bi-Whittaker categories 
   and the category of $\rW$-equivariant 
central  complexes on 
   torus introduced in our previous work on Braverman-Kazhdan conjecture.

}

\author{Tsao-Hsien Chen}
\end{abstract}

\maketitle
\setcounter{tocdepth}{1}
\tableofcontents

\section{Introduction}
\subsection{}
Let $k$ be an algebraically closed field of characteristic $\on{char}(k)=p>0$
and let $\ell$ be a prime number different from $p$. 
Let $G$ be a connected reductive group over $k$. Let $T$ be a maximal tours of $G$ and let 
$\rW=N(T)/T$ be the Weyl group.
Denote by $D_\rW(T)$  the bounded derived category of $\rW$-equivaraint $\ell$-adic complexes on $T$. 
In our previous work \cite{C1,C2}, we introduced and studied 
a certain 
monoidal subcategory 
$D_\rW^\circ(T)\subset D_\rW(T)$, called the category of $\rW$-equivariant
central complexes on $T$.
The category $D_\rW^\circ(T)$ can be viewed as the $\ell$-adic counterpart of  
the category of modules over the quantum Toda lattices or nil-Hecke algebras, see \cite{G,L}.

In this paper, we show that the category $D_\rW^\circ(T)$ of central complexes admits 
functorial transfers with respect to 
morphisms 
between dual groups (Theorem \ref{transfer:central complexes}):  
given two connected reductive groups $G'$ and $G$
over $k$ and a 
 morphism $\hat\rho:\hat G'\to \hat G$ 
 between their complex dual groups, 
 there is a canonical monoidal transfer morphism
\[\Phi_{\hat\rho}:D^\circ_\rW(T)\to D^{\circ}_{\rW'}(T')\]
which is functorial with respect to  compositions of morphisms between dual groups. This is a bit interesting since in general the dual map on torus 
$\rho_{T,T'}:T\to T'$ is not compatible with 
actions of the Weyl groups $\rW$ and $\rW'$
(see Section \ref{dual map}).

Combining with the work of Bezrukavnikov-Deshpande \cite{BT}, we show that  
the $\ell$-adic bi-Whittaker category $\on{Wh}_G$ of $G$,
the category $D^\circ_G(G)$ of vanishing  complexes, and the 
 subcategory of stable complexes
$D^{st}_G(G)\subset D^\circ_G(G)$ (introduced in this note, see Definition \ref{stabble complexes}), 
admit functorial transfers with respect to morphisms bewteen dual groups (Theorem \ref{transfer:whittaker} and Theorem \ref{Trans for vansihing}).

In the case when $G$ is defined over a finite field, 
we show that the 
transfer morphism for  stable complexes 
on $G$
provides a geometrization of the transfer map
for stable functions on the  finite reductive group $G^F$ (Theorem \ref{geometrization}). As an application, we give a new geometric proof of 
Laumon-Letellier's fromula for transfer maps for stable functions \cite{LL1}
(Corollary \ref{LL formula}).

\subsection{Acknowledgement}
 A part of the
work was done while the author visited the Collège de France and he would like to thank
the research institute its inspiring environment.
The author would like to 
thank R. Bezrukavnikov,  V.Ginzburg and B.C. Ng\^o for useful discussion.
The research  is supported
by NSF grant DMS-2143722.

\section{Central complexes}

\subsection{}
Let $T$ be a maximal torus of $G$ and 
$B$ be a Borel subgroup containing $T$ with 
unipotent radical $U$. 
Let $N=N_G(T)$ is the normalizer of $T$ in $G$ and we denote by
$\rW=\rW_G=N/T$ the Weyl group.
Let 
$\mathcal C(T)(\bar{\mathbb Q}_\ell)$ be the set consisting of 
characters of the tame \'etale fundamental group $\pi_1(T)^t$.
For any $\chi\in\mathcal C(T)(\bar{\mathbb Q}_\ell)$ we denote by 
$\mL_\chi$ the corresponding tame local system on
$T$ (a.k.a the Kummer local system associated to $\chi$).
The Weyl group $\rW$ acts naturally on $\calC(T)(\barQ)$ and for any $\chi\in\calC(T)(\barQ)$, 
we denote by $\rW_\chi'$ the stabilizer of 
$\chi$ in $\rW$ and 
$\rW_\chi^\circ\subset\rW_\chi$, the subgroup of $\rW_\chi'$ generated by those reflections 
$s_\alpha$ satisfying the following property:
we have $(\check\alpha)^*\mL_\chi\is\barQ$, where 
$\check\alpha:\bG_m\to T$ is the coroot associated 
to $\alpha$. The group
$\rW^\circ_\chi$ is a normal subgroup of $\rW_\chi$ and,
in general, we have $\rW_\chi^\circ\subsetneq\rW_\chi$.

Denote by
$ D_\rW(T)$ (resp. $D_N(T)$) symmetric monoidal category 
the $\rW$-equivariant (resp. $N$-equivariant) bounded derived category of constructible 
$\ell$-adic complexes on $T$ with monidal structure given by 
the $!$-convolution
$\mF\star\mF'=m_!(\mF\boxtimes\mF')$
where $m:T\times T\to T$ is the multiplication.
Note that there is an equivalence between $D_N(T)
$ and the category 
$D_\rW(\frac{T}{T})$ of 
$\rW$-equivariant complexes on the quoteint stack $\frac{T}{T}$ (see, e.g., \cite[Section 6.1]{LL1}).
We denote by $\on{sign}_\rW\in D_\rW(T)$ the 
constant local system on $T$ with the $\rW$-equivariant structure given by the sign character of $\rW$.

For any $\mF\in D_\rW(T)$
and $\chi\in\calC(T)(\barQ)$, the 
$\rW$-equivariant structure on $\mF$
together with
the natural $\rW_\chi'$-equivivariant structure on $\mL_\chi$
give rise to an
action of 
$\rW_\chi$ on the
\'etale cohomology groups
$H_c^*(T,\mF\otimes\mL_\chi)$ (resp. $H_c^*(T,\mF\otimes\mL_\chi)$).
In particular, we get an action of the subgroup 
$\rW_\chi^\circ\subset\rW_\chi$ on the cohomology groups above.

 Recall the notion of 
 central complexes and strongly central complexes introduced in \cite{C1,C2, BT}.
\footnote{This notion agrees with that defined in \emph{loc. cit.} after twisting the $\rW$-equivariance structure by the sign character of $\rW$.}
\begin{definition}\label{def of central}
(1) A $\rW$-equivariant complex $\mF\in D_\rW(T)$ 
is called \emph{central} (resp. \emph{strongly central}) if for any tame character $\chi\in\calC(T)(\barQ)$, the group
$\rW^\circ_\chi$ (resp. $\rW_\chi$) acts trivially on  $H_c^*(T,\mF\otimes\mL_\chi)$. 
We denote by $D^\circ_\rW(T)\subset D_\rW(T)$ the monoidal subcategory of $\rW$-equivariant central complexes on 
$T$ and $D^{st}_\rW(T)$ the  monoidal subcategory of strongly central complexes.

(2) A $N$-equivariant complex $\mF\in D_N(T)$ is  called \emph{central} (resp. \emph{strongly central}) if its image in 
$D_W(T)$ under the natural forgetful functor is central (resp. strongly central). We denote by $D^\circ_N(T)\subset D_N(T)$ the monoidal subcategory of 
$N$-equivariant central complexes on 
$T$ and $D^{st}_N(T)$ the  monoidal subcategory of $N$-equivariant strongly central complexes.

(3) We denote by 
$D^\circ_\rW(T)^\heartsuit$ (resp. $D^\circ_N(T)^\heartsuit$) and $D^{st}_W(T)^\heartsuit$ (resp. $D^{st}_N(T)^\heartsuit$) for the  categories of central and strongly central perverse sheaves on $T$.

\end{definition}

\begin{example}\label{example of central compelxes}
(1) In the case $G=SL_2$,
$\rW^\circ_\chi=1$  if and only if $\chi=1$
and thus 
a complex $\mF\in D_\rW(T)$ is central if and only if $\rW$ acts trivially on $H_c^*(T,F)$.
For example, any non-trivial tame local system 
$\mL_\chi$ satsifying 
$H^*_c(T,\mL_\chi)=0$ and hence is central.

Note that the non-trivial tame local system 
$\mF=\mL_\chi$ of order two $\chi^2=1$ is central but is not strongly central as $\rW_\chi=\rW$
and $H_c^*(T,\mF\otimes\mL_\chi)^\rW=
H_c^*(T,\bar{\mathbb Q}_\ell)^\rW\neq H_c^*(T,\bar{\mathbb Q}_\ell)$.

(2) Let $G=GL_n$. Consider the trace map 
$\on{tr}_T:T\is\bG_m^n\to\mathbb A^1, (t_1,...,t_n)\to \sum_{i=1}^n t_i$.
The pullback $\on{tr}_T^*(\mL_\psi)$ of the 
Artin–Schreier sheaf $\mL_\psi$ on $\mathbb A^1$
is naturally a $\rW$-equivaraint complex and 
it was shown in \cite{C1} that 
$\on{tr}_T^*(\mL_\psi)\otimes\on{sign}_\rW$
is a central complex on $T$.

\end{example}

\begin{remark}\label{center}
If the center of $G$ is connected, then 
$\rW^\circ_\chi=\rW_\chi$ and hence 
there is no differnce between central and stongly central complexes.
\end{remark}

\section{Dual morphisms}\label{dual maps}
\subsection{}\label{construction}
Let $G$ and $G'$ be two connected reductive group over $k$ and 
let $\hat\rho:\hat G'\to\hat G$ be a morphism between their Langlands dual groups. 
Let $T'$ be a maximal  $G'$ with 
dual torus $\hat T'$ and we denote by
$\hat Z=\rho^\vee(\hat T')\subset \hat G$ its image  which is a torus in $\hat G$. Let $\hat L=Cen_{\hat G}(\hat Z)$ be the Levi subgroup of centralizers of $\hat Z$ in $\hat G$. 
Let $L\subset  G$ be the dual Levi of $\hat L$ and we pick a maximal torus  
$T\subset L$ such that the dual torus 
$\hat T\subset \hat L$ containing $\hat Z$.

The map $\hat\rho$ restricts a map
\[\hat\rho_{T,T'}:=\hat\rho|_{\hat T'}:\hat T'\twoheadrightarrow \hat Z\hookrightarrow \hat T\]
between the dual maxmial torus
and hence, by 
duality for torus, a morphism
\begin{equation}\label{dual map}
    \rho_{T,T'}:T\twoheadrightarrow Z\hookrightarrow T'
\end{equation}
between maximal torus.
On the other hand,
we have a map 
\[\rho_{\rW',\rW}:\rW'\is N_{\hat G'}(\hat T')/\hat T'\stackrel{\hat\rho}\longrightarrow 
N_{\hat G}(\hat Z)/\hat L=
N_{\hat G}(\hat L)/\hat L\is N_{\rW}(\rW_{L})/\rW_{L} \]
where $N_{\rW}(\rW_{L})$ is the normailizer of the Weyl group $\rW_{L}=N_L(T)/T$ in $\rW$. It follows from the construction that, 
for any $w'\in \rW'$ and any lift 
$w\in N_{\rW}(\rW_L)$
of $\rho_{\rW',\rW}(w')\in N_{\rW}(\rW_{L})/\rW_{L}$, there is a commutative diagram
\begin{equation}\label{lifting}
\xymatrix{T\ar[r]^{\rho_{T,T'}}\ar[d]^{w}&T'\ar[d]^{w'}\\
T\ar[r]^{\rho_{T,T'}}&T'}.
\end{equation}

\begin{example}\label{normal morphism}
    (1) Assume $\hat G'=\hat M$ is a Levi subgroup of $\hat G$ 
    and 
    $\hat\rho:\hat M\to\hat G$ is the inclusion. We have $T'=Z=T=L$, $\rho_{T,T'}=\text{id}:T\to T$, and 
    $\rho_{\rW',\rW}:\rW_M\to\rW$ is the  embedding.
    
    (2) (Normal morphism) Let $\hat\rho:\hat G'\to\hat G$ be a normal morphism,
     that is, when $\hat\rho(\hat G')$ is a normal subgroup. 
        Then we have a natural lift 
        $\tilde\rho_{\rW',\rW}:
        \rW'\to N_{\rW}(\rW_L)$
        of $\rho_{\rW',\rW}$.
        It suffices to deal with the cases when $\hat\rho$ is surjection or is an embedding of closed connected normal reductive subgroup.
       The first case is clear, since $\hat Z=\hat T=\hat L$
    and $\rW_L=e$. 
    In the second case, 
    There is an isogeny $\hat\rho_1:\hat G'_1:=\hat G'\times S\to \hat G$ extending $\hat\rho$, where $S$
    is a connected reductive subgroup. 
    Let $T_S\subset S$ be a maxmial torus of $S$. Then  
$\hat L'_1=Cen_{\hat G'_1}(\hat T')=\hat T'\times S$
    is a  Levi subgroup of $\hat G_1'$ with maximal torus  $\hat T_1'=\hat T'\times T_S$.
    The image
     $\hat L=\hat\rho_1(\hat L_1)$
     is a Levi subgroup with maximal torus 
    $\hat T=\hat\rho_1(\hat T_1') $
and the map $\hat\rho
$ restricts to a map
\[N_{\hat G'}(\hat T')\subset N_{\hat G'}(\hat T'_1)\cap N_{\hat G'}(\hat L'_1)\stackrel{\hat\rho}\to N_{\hat G}(\hat T)\cap N_{\hat G}(\hat L)\]
which induces the desired lift
$\rW'\to N_\rW(\rW_L)$.

Note that, according to \cite[Section 2.5]{B} or \cite[Proposition 3.1.8]{LL1}, there is a dual morphism $\rho:G\to G'$ such that 
$\rho|_{T}=\rho_{T,T'}:T\to T'$.
    \end{example}

The following lemma will be used later (for the case of when the center of $G$ is disconnected).
  \begin{lemma}\label{z-extension}
  Let $w'$ and $w$ be as above.
Let $\chi'$ be a tame character of $T'$
and let $\chi=\rho_{T,T'}^*\chi'$.
We have
i) $w'\in\rW'_{\chi'}$ implies $w\in\rW_\chi$.
ii) $w'\in(\rW'_{\chi'})^\circ$ implies $w\in\rW^\circ_\chi$.

  \end{lemma}  
\begin{proof}
Part i) follows from 
 the diagram ~\eqref{lifting}. 
For part ii), if the center of $G$ is connected then we have $\rW^\circ_\chi=\rW_\chi$
(Remark \ref{center}) and the claim follows from part i).
To deal with the case of disconneted center, we use the 
notion of $z$-extension in \cite{K}.
According to 
\cite[Lemma 2.4.4]{K},
we can find finite central extensions
$1\to \hat C\to \hat H\to \hat G\to 1$
and $1\to \hat C'\to \hat H'\to \hat G'\to 1$
satisfying i) the derived subgroups  $\hat H_{dr},\hat H'_{dr}$
are simply connected
ii) there is a commutative diagram 
\begin{equation}\label{central extension}
    \xymatrix{\hat H\ar[r]^{\hat\phi}\ar[d]&\hat H'\ar[d]\\
\hat G'\ar[r]^{\hat \rho}&\hat G}.
\end{equation}
Let $H,H'$ be the connected reductive groups over $k$ with dual groups $\hat H,\hat H'$. Note that $H,H'$ have connected centers since the derived subgroups of their dual groups are simply connected.
Let $T_H, T_{H'}$ be the maximal torus of $H,H'$ such that the dual torus 
$\hat T_H,\hat T_{H'}$ are the 
pre-imaiges of $\hat T,\hat T'$.
There are natural identifications of Weyl groups of $G$ and $H$ (resp. $G'$ and $H'$)
and the diagram~\eqref{central extension} and~\eqref{lifting} 
give rise to 
commutative diagrams 
\[\xymatrix{ T\ar[r]^{\rho_{T,T'}}\ar[d]_{\pi}&T'\ar[d]^{\pi'}\\ T_H\ar[r]^{\phi_{T_H,T_{H'}}}&T_{H'}}\ \ \ \ \ \ \ \ \ \ \xymatrix{ T_H\ar[r]^{\phi_{T_H,T_{H'}}}\ar[d]_{w}&T_{H'}\ar[d]^{w'}\\ T_H\ar[r]^{\phi_{T_H,T_{H'}}}&T_{H'}}.\]
where $\pi,\pi'$ are Weyl group equivariant central isogenies. One can find a tame character   $\chi_{H'}$ of $T_{H'}$ such that $\chi'=(\pi')^*\chi_{H'}$. Let  $\chi_H=(\phi_{T_H,T_H'})^*\chi_{H'}$ and 
the left diagram above implies 
$\chi=\pi^*\chi_H$.   
By \cite[Lemma 2.9]{BT}, we have 
$\rW'_{\chi_{H'}}=(\rW^{'}_{\chi'})^\circ$ and $\rW_{\chi_H}=\rW_\chi^\circ$.
Thus if $w'\in\rW'_{\chi_{H'}}=(\rW'_{\chi'})^\circ$ then
the right diagram above implies 
$w\in \rW_{\chi_H}=\rW^\circ_\chi$.
The desried claim follows.
\end{proof}

\section{Functorial transfers}
In this section we preserve the setups in Section \ref{dual map}.
\subsection{Transfer for $\rW$-equivariant central complexes}
Let $\mF\in D_\rW(T)$ be a $\rW$-equivariant complex on $T$
and let $\mF'=(\rho_{T,T'})_!(\mF)\in D(T')$. 
For any $w'\in\rW'$ and a lift 
$w\in N_\rW(\rW_L)$ of $\rho_{\rW',\rW}(w')\in N_\rW(\rW_L)/\rW_L$, 
the $\rW$-equivaraint structure on $\mF$  gives rise to an isomorphism 
\[b_{w}:(w')^*\mF'\is 
(w')^*(\rho_{T,T'})_!(\mF)\is(\rho_{T,T'})_!( w)^*\mF\is (\rho_{T,T'})_!\mF\is\mF'.\]

\begin{lemma}\label{W-equ}
Assume $\mF\in D^\circ_\rW(T)$ is a $\rW$-equivariant central complex on $T$.
Then the isomorphism $b_{w}$
above depends only on $w'$ (that is, it
is independent of the choice of the lift $w$) and the collection 
$\{a_{w'}:=b_{w}:(w')^*\mF'\is\mF'\}_{w'\in\rW' }$
defines a $\rW'$-equivariant structure on $\mF'$. 
    
\end{lemma}
\begin{proof}
Note that the map $\rho_{T,T'}:T\to T'$ is 
$\rW_L$-equivariant, where $\rW_L$-acts trivially on $T'$,
and 
it suffices to show that  
the $\rW_L$-action on $\mF'=(\rho_{T,T'})_!(\mF)$ is trivial. For this, it is enough to show that for any
simple reflection $s_\alpha$ in $\rW_L$ 
the $(-1)$-summand 
$\mF'^{s_\alpha=-1}=0$ is zero.
By the conservativity of Mellin transfrom \cite[Prop 3.4.5]{GL}, we reduce to show that 
\[H^*_c(T',\mF'^{s_\alpha=-1}\otimes\mL)=0\]
for any tame local system $\mL$ on $T'$. Note that the projection formula induces a $\rW_L$-equivariant isomorphism
\begin{equation}\label{projection}
    H^*_c(T,\mF\otimes\rho_{T,T'}^*\mL)\is H^*_c(T,(\rho_{T,T'})_!\mF\otimes\mL)\is H^*_c(T',\mF'\otimes\mL)
\end{equation}
On the other hand, 
by \cite[Section 2.15]{S}, the pull back along the surjection 
$T\to Z$ induces an isomorphism between  
\begin{equation}\label{char of Z}
    X^*(Z)\is\{\lambda\in X^*(T)|\langle\lambda,\check\alpha\rangle=0\ \ \text{for any root $\alpha$ of $L$}\}
\end{equation}
Since $\rho_{T,T'}:T\to T'$ factors through $T\to Z$, we see that 
for any root $\alpha$ of $L$ we have
\[(\check\alpha)^*\rho_{T,T'}^*\mL\is\bar{\mathbb Q}_\ell\]
and the assumption that $\mF$ is central implies the desired vanishing 
\[H^*_c(T',\mF'^{s_\alpha=-1}\otimes\mL)\is 
H^*_c(T',\mF'\otimes\mL)^{s_\alpha=-1}\is H^*_c(T,\mF\otimes\rho_{T,T'}^*\mL)^{s_\alpha=-1}=0\]
\end{proof}

\begin{lemma}\label{central}
Let $\mF\in D^\circ_W(T)$ be a $\rW$-equivariant central complex
and let $\mF'=(\rho_{T,T'})_!\mF\in D_{W'}(T')$ be 
the $\rW'$-equivariant complex as in Lemma \ref{W-equ}. 
Then $\mF'\in D^\circ_{W'}(T')$ is central. If $\mF\in D^{st}_W(T)$ is strongly central, then $\mF'\in D^{st}_{W'}(T')$
is also strongly central.
\end{lemma}
\begin{proof}
Let $\chi'$, $\chi$,
$w'$, and $ w$
be as in Lemma \ref{z-extension}.
Assume $w'\in\rW_{\chi'}$ (resp. $w'\in\rW^\circ_{\chi'})$. Then
 Lemma \ref{z-extension} implies 
 $w\in\rW_\chi$ (resp. $ w\in\rW^\circ_\chi$)
 and diagram~\eqref{lifting}
induces a commutative diagram
\[\xymatrix{H_c^*(T,\mF\otimes\mL_{\chi})\ar[r]^{\simeq}\ar[d]^{w^*}&H_c^*(T',\mF'\otimes\mL_{\chi'})\ar[d]^{(w')^*}\\
H_c^*(T,\mF\otimes\mL_{\chi})\ar[r]^{\simeq}&H_c^*(T',\mF'\otimes\mL_{\chi'})}\]
where the horizontal map is the isomorphism induced by the projection formula~\eqref{projection}. If $\mF$  central (resp. strongly central) then
$w^*$ is the identity and the diagram above implies 
$(w')^*$ is the identity
and hence $\mF'$ is central (resp. strongly central). The lemma follows.

\end{proof}

Combining Lemma \ref{W-equ} and Lemma \ref{central} we obtain
\begin{thm}\label{transfer:central complexes} 
Given a morphism 
$\hat\rho:\hat G'\to\hat G$,  
the dual morphism $\rho_{T,T'}:T\to T'$ in~\eqref{dual map} induces a monoidal functors
\[\Phi_{\hat\rho}:=(\rho_{T,T'})_!: D^\circ_\rW(T)\to D^\circ_{\rW'}(T')\ \  (resp.\ \ \  \Phi_{\hat\rho}=(\rho_{T,T'})_!:D^{st}_\rW(T)\to D^{st}_{\rW'}(T'))\]
between the category of $\rW$-equivariant central complexes (resp. strongly central complexes)
which is functorial with respect to compositions of maps between dual groups, that is, 
for any pair of morphisms 
    $\hat\rho':\hat G^{''}\to\hat G'$ and 
    $\hat\rho:\hat G'\to\hat G$
    there is a canoincal isomorphism of monoidal fucntors
    $\Phi_{\hat\rho'}\circ\Phi_{\hat\rho}\is\Phi_{\hat\rho\circ\hat\rho'}$.
\end{thm}

\begin{example}\label{gamma sheaves on T}
Consider the case when $\hat\rho:\hat G'\to \hat G=GL_n$ 
is a representation $\hat G'$. Let $\on{tr}_T^*\mL_\psi\otimes\on{sign}_\rW\in D_\rW^\circ(T)$ be the central complex in Example \ref{example of central compelxes} (2) for $G=GL_n$.
The transfer $\Phi_{T',\hat\rho}=\Phi_{\hat\rho}(\on{tr}_T^*\mL_\psi\otimes\on{sign}_\rW)\in D_{\rW'}^\circ(T')$ is the $\gamma$-sheaves on $T'$  in \cite{BK}.
\end{example}

\subsection{Transfer for bi-Whittaker categories}
Let $\mL_\psi$ be a non-degenerate multiplicative local system on $U$
and let $\on{Wh}(G)=D_{U\times U.\mL_\psi\boxtimes\mL_\psi}(G)$ be the bi-Whittaker category of 
bounded $(U,\mL_\psi)$ bi-equivariant 
constructible complexes on $G$.
By \cite[Theorem 1.4]{BT}, there is a $t$-exact monoidal equivalence 
\[\xi_G:\on{Wh}_G\is D^\circ_\rW(T).\]
Then Theorem \ref{transfer:central complexes} implies the following
\begin{thm}\label{transfer:whittaker}
Given a morphism 
$\hat\rho:\hat G'\to\hat G$ between dual groups, the composition
\[\on{Wh}_{\hat\rho}={\xi_{G'}}^{-1}\circ\Phi_{\hat\rho}\circ\xi_G:\on{Wh}_G\to\on{Wh}_{G'}\]
define a monoidal transfer functor between 
bi-Whittaker categories
which is functorial with respect to compositions of maps between dual groups, that is, 
for any pair of morphisms 
    $\hat\rho':\hat G^{''}\to\hat G'$ and 
    $\hat\rho:\hat G'\to\hat G$
    there is a canoincal isomorphism of monoidal fucntors
    $\on{Wh}_{\hat\rho'}\circ\on{Wh}_{\hat\rho}\is\on{Wh}_{\hat\rho\circ\hat\rho'}: \on{Wh}_{G}\to \on{Wh}_{G''}$

\end{thm}

\subsection{Transfer for vanishing complexes}
For any parabolic subgroup 
$P$ with a Levi subgroup $L$
we have the  restriction functor
$\on{Res}_{L,P}^G=   q_!p^*:D_G(G)\to D_L(L)$, where 
\[\frac{G}{G}\stackrel{p}\longleftarrow\frac{P}{P}\stackrel{q}\longrightarrow\frac{L}{L}.\]
It admits a right adjoint 
$\on{Ind}_{L\subset P}^G:D_L(L)\to D_G(G)$
called the induction functor.
We also have the Harish-Chandra functor
\[\on{HC}=a_!\circ\on{For}:D_G(G)\longrightarrow 
D_B(G)\Longrightarrow D_B(G/U)\]
where $a:\frac{G}{B}\to \frac{G/U}{B}$
is the natural quotient map.

\begin{definition}
    A complex $\mF\in D_G(G)$ is a called a \emph{vanishing complex}
    if $\on{HC}(\mF)$ is supported on 
    $T=B/U\subset G/U$.
    We denote by 
    $D^\circ_G(G)$ the category of vanishing complexes
    and $D^\circ_G(G)^{\heartsuit}$
    the category of vanishing perverse sheaves.
\end{definition}

Vanishing complexes were  introduced and studied in \cite{C1,C2}
 on the Braverman-Kazhdan conjecture.
The main result in \emph{loc. cit.} implies  that 
 for any  $N$-equivaraint 
 central perverse  sheaf
 $\mF\in D^\circ_N(T)^{\heartsuit}\subset D_\rW(\frac{T}{T})^\heartsuit$ on $T$, its induction 
 $\on{Ind}_{T\subset B}^G(\mF)$
 carries a $\rW$-action and the $\rW$-invariant summand 
$\on{Ind}_{T\subset B}^G(\mF)^\rW\in D^\circ_G(G)^{\heartsuit}$ is a vanishing perverse sheaf. 
It was conjectured in \cite[Conjecture 1.9.]{C2} that
the resulting functor
\[\on{Ind}_{T\subset B}^G(-)^\rW:C_N(T)^\heartsuit\to C_G(G)^\heartsuit\]
 is an equivalence of abelian category.  
 The conjecture was proved
in \cite{BT}. In fact, they proved a more strong results at the level of triangulated categories:
\begin{thm}\cite[Corollary 3.2]{BT}\label{ind_N res _N}
(1)
There is natural $\rW$-action on 
the functor 
$\on{Ind}_{T\subset B}^G\circ\on{For}:D^\circ_N(T)\subset D_N(T)\to D_T(T)\to D_G(G)$
and the taking $\rW$-invariant factor 
defines a $t$-exact braided monoidal equivalence 
\[\on{Ind}_N:=\on{Ind}_{T\subset B}^G(-)^\rW\circ\on{For}: D^\circ_N(T)\stackrel{\simeq}\longrightarrow D^\circ_G(G)\]
between $N$-equivariant central complexes on $T$
and the category of vanisihing complexes on $G$.

(2) 
For any vanishing complex $\mF\in D^\circ_G(G)$, the restriction  
$\mF_T=\on{Res}_{T\subset B}^G(\mF)\in D_T(T)$ carries a natural 
$\rW$-equivariant structure. Moreover, the resluting 
$N$-equivaraint complex
$\mF_T'\in D^\circ_N(T)\subset D_N(T)$ is central and the assignment 
$\mF\to \mF_T'$ defines an inverse equivalence
\[\on{Res}_N:D^\circ_G(G)\stackrel{\simeq}\longrightarrow D^\circ_N(T)\]
of $\on{Ind}_\rW$.
\end{thm}

\begin{remark}\label{sign}
(1) 
The $\rW$-action in
part (1) differs from the one in \cite{C2,LL1}
by the sign character of $\rW$.

(2) As observed in \cite{LL2}, the induction 
$\on{Ind}_{T\subset B}^G\circ\on{For}:D_N(T)\to D_T(T)\to D_G(G)$
on the whole category might not carry a 
$\rW$-action and the restriction 
$\on{Res}_{T\subset B}^G:D_G(G)\to D_T(T)$
does not admit a lifting to 
$D_N(T)$. The theorem above in particular says that, when restricts to the subcategories of 
central complexes or vanishing complexes, we do have a
canonical $\rW$-action and a lifting functor.

\end{remark}

Consider 
and following monoidal functors
\[\on{Ind}_\rW=\on{Ind}_N\circ f^*:D^\circ_\rW(T)\to D^\circ_N(T)\is D^\circ_G(G)\]
\[\on{Res}_\rW=\pi^*\on{Res}_N:D^\circ_G(G)\is D^\circ_N(T)\to D^\circ_\rW(T)\]
where 
$f:\frac{T}{T}\to T$ and $\pi:T\to\frac{T}{T}$
are the natural projection maps.
Note that Theorem \ref{ind_N res _N} implies
\begin{equation}\label{res-ind}
\on{Res}_\rW\circ\on{Ind}_\rW\is
\pi^*\on{Res}_N\circ\on{Ind}_N \circ f^*\is\pi^*f^*\is\on{id}.
\end{equation}

Motiviated by the work \cite{C3}, we introduce the notion of 
\emph{stable complexes} on $G$:

\begin{definition}\label{stabble complexes}
A vanishing complex $\mF\in D^\circ_G(G)$ is called a stable complex if its restriction $\on{Res}_\rW(\mF)\in D^{st}_\rW(T)\subset D^\circ_\rW(T)$ is strongly central.
We denote by $D^{st}_G(G)$ the full subcategory of stable complexes on $G$.
    
\end{definition}

\begin{thm}\label{Trans for vansihing}
Given a morphism $\hat\rho:\hat G'\to \hat G$ between dual groups,  the composition
\[\on{T}_{\hat\rho}:=\on{Ind}_{\rW'}\circ\Phi_{\hat\rho}\circ\on{Res}_\rW:D^\circ_G(G)\to D^\circ_\rW(T)\to D^\circ_{\rW'}(T')\ra D^\circ_{G'}(G')\]
defines a functorial monoidal transfer functor
between the categories of vanishing complexes  
satsifying the following properties
\begin{enumerate}
\item it restricts to a monoidal functor
$\on{T}_{\hat\rho}: D^{st}_G(G)\to D^{st}_{G'}(G')$
between stable complexes
    \item  
    for any pair of morphisms 
    $\hat\rho':\hat G^{''}\to\hat G'$ and 
    $\hat\rho:\hat G'\to\hat G$
    there is a canoincal isomorphism of monoidal fucntors
    $\on{T}_{\hat\rho}\circ\on{T}_{\hat\rho'}\is\on{T}_{\hat\rho\circ\hat\rho'}: D^\circ_{G^{}}(G^{})\to D^{\circ}_{G''}(G'')$
    \item 
    the averaging fucntor
\[\on{Av}_{\psi}:D^{\circ}_G(G)\to\on{Wh}_G\ \ \ (resp.
\ \ \on{Av}_{\psi'}:D^{\circ}_{G'}(G')\to\on{Wh}_{G'})\]
    with respect to 
    a  non-degenerate local system 
    $\mL_\psi$ 
    on $U$ (resp.  $\mL_{\psi'}$ on $U'$),    
    fits into the following commutative diagram
\[\xymatrix{D^{\circ}_G(G)\ar[r]^{\on{T}_{\hat\rho}}\ar[d]^{\on{Av}_{\psi}}&D^\circ_{G'}(G')\ar[d]^{\on{Av}_{\psi'}}\\
\on{Wh}_G\ar[r]^{\on{Wh}_{\hat\rho}}&\on{Wh}_{G'}}.\]
\end{enumerate}

\end{thm}
\begin{proof}
Part (1) and (2) follow from
Theorem \ref{ind_N res _N} and 
Theorem \ref{transfer:central complexes}.
For part (3), we need to check the following daigram is commutative
\[\xymatrix{D_G^\circ(G)\ar[r]^{\on{Res}_\rW}\ar[d]^{\on{Av}_\psi}&D_\rW^\circ(T)\ar[r]^{\on{T}_{\hat\rho}}\ar[d]^{\on{id}}&D_{\rW'}^\circ(T')\ar[r]^{\on{Ind}_\rW}\ar[d]^{\on{id}}&D^\circ_{G'}(G')\ar[d]^{\on{Av}_{\psi'}}\\
\on{Wh}_G\ar[r]^{\xi_G}&D_\rW^\circ(T)\ar[r]^{\on{T}_{\hat\rho}}&D_{\rW'}^\circ(T')\ar[r]^{\xi_{G'}^{-1}}&\on{Wh}_{G'}}.\]
\cite[Propoisition 5.2]{BT} implies the 
 commutativity of 
the left square diagram 
and \cite[Corollary 5.4 (iii)]{BT} implies the commutativity of the right square diagram.
The desired claim follows.
\end{proof}

\begin{example}\label{gamma sheaves on G}
Let $\hat\rho:\hat G'\to \hat G=GL_n$ be a representation.
Let $\on{tr}:GL_n\to\mathbb A^1$ be the trace map.
In \cite{C1,C2}, we shown that $\on{tr}^*\mL_\psi\in D_G^\circ(G)$
is a vanishing complex and its transfer 
$\Phi_{G',\hat\rho}=\on{T}_{\hat\rho}(\on{tr}^*\mL_\psi)\in D_{G'}^\circ(G')$
is the $\gamma$-sheaf on $G'$ in \cite{BK}.
The fact that $\Phi_{G',\hat\rho}$ is a vanishing complex is the content of 
\cite[Conjecture 6.5]{BK}.

\end{example}

\section{Transfer for stable functions}
We shall show that
the transfer functor between stable complexes
on reductive group over finite fields 
provides a geometrization of the transfer map between stable functions on finite reductive groups.
As an application, we 
give a new geometric proof of Laumon-Letellier's formula for 
  the transfer map in \cite{LL1}.
\subsection{Stable functions}
We assume $k=\overline{\mathbb F}_q$  
and $G$ and $T$ are defined over $\mathbb F_q$. Let $F:G\to G$ be the (geometric) Frobenius endomorphism.
The Frobenius endomorphism induces 
an endomorphism $\hat F:\hat G\to \hat G$ of the dual 
group and the dual torus $\hat T$ is stable under $\hat F$. 
Let $G^F$ be the corresponding finite reductive group. 
Let $\on{Irr}(G^F)$ be the set of irreducible $\overline{\mathbb Q}_\ell$-characters of $G^F$
and let $(C(G^F),\star)$ be the ring of $\overline{\mathbb Q}_\ell$-class functions on $G^F$ with multiplication given by convolution of functions
$f\star f'(g)=\sum_{h\in G^F} f(gh^{-1})f'(h)$.

According to \cite[Section 5.6]{DL}, the set of $\hat F$-stable semisimple conjugacy classes of $\hat G$ are in bijection with the set 
$(\hat T//\rW)^{\hat F}$ of $\hat F$-fixed points of the GIT quotient $\hat T//\rW$. The main results in \cite{DL} imply that there is a surjective map
\[\mathcal L:\mathrm{Irr}(G^F)\to(\hat T//\rW)^{\hat F}\]
For each class function $f\in C(G^F)$ there exists 
a $\gamma$-function
$\gamma_f:\on{Irr}(G^F)\to\overline{\mathbb Q}_\ell$
characterized by the equation
\[
f\star\chi=\gamma_f(\chi)\chi\]
for all $\chi\in\on{Irr}(G^F)$.

\begin{definition}
A class function $f\in C(G^F)$ is called \emph{stable} if 
the corrsponding $\gamma$-function 
factors through the Deligne-Lusztig map
\[\gamma_f:\on{Irr}(G^F)\stackrel{\mathcal L}\to(\hat T//\rW)^{\hat F}\to\overline{\mathbb Q}_\ell.\]
That is, for any $\chi,\chi'\in\mathrm{Irr}(G^F)$ 
we have 
$\gamma_f(\chi)=\gamma_f(\chi')\ \ \text{if}\ \ \mathcal{L}(\chi)=\mathcal{L}(\chi')$.
We denote by $C^{st}(G^F)$ the space of stable fucntions on $G^F$
\end{definition}

Note that for any stable function $f\in C^{st}(G^F)$
one can view the corresponding $\gamma$-function as 
a fucntion 
\[\gamma_f:(\hat T//\rW)^{\hat F}\to\overline{\mathbb Q}_\ell\]
and the assignment $f\to \gamma_f$
defines a ring isomorphism
\begin{equation}\label{description}
    C^{st}(G^F)\is\overline{\mathbb Q}_\ell[(\hat T//\rW)^{\hat F}]
\end{equation}
where the right hand side is the ring of $\overline{\mathbb Q}_\ell$-functions on 
$(\hat T//\rW)^{\hat F}$
and the ring structures on $C^{st}(G^F)$ is given by 
convolution.
For any $w\in\rW$, let $F_w=w\circ F:T\to T$ be the $w$-twisted Frobenius endomorphsm of $T$.
For a
character $\theta$ of $T^{F_w}$, let
$DL_\theta$ be the Deligne-Lusztig virtual character attached to $\theta$, and let 
$[\theta]\in (\hat T//\rW)^{\hat F}$ be the 
semi-simple $\hat F$-stable conjugacy classes in $\hat G$ associated to
$\theta$ (\cite[Section 5.6]{DL})

\begin{lemma}
The isomorphism~\eqref{description} is characterized by the property that for any 
$w\in\rW$ and any
character 
$\theta$ of $T^{F_w}$, we have 
\[f\star DL_\theta=\gamma_f([\theta])DL_\theta.\]
\end{lemma}

We shall show that the category stable complexes 
on $G$ provide a geometrization of stable functions on $G^F$.
Note that 
all the geometric objects introduced before 
are defined over $\mathbb F_q$ and we denote by 
$D_G(G)^F$, $D^\circ_G(G)^F$, etc,  the corresponding categories of $F$-equivariant objects.
For any $w\in\rW$ we let
$D(T)^{F_w}$ be the category of $F_w$-equivariant object in $D(T)$.
For any complex $\mF$ and an integer $n$, we write $F\langle n\rangle=\mF[n](n/2)$.

Let $\mF\in D_G(G)^F$ and denote by 
$\mF_T=\on{Res}_\rW(\mF)\in D_\rW(T)^F$.
Note that 
for any $w\in\rW$,
we have an isomorphism $F_{w}^*\mF_{T} \is F^*(w)^*\mF_{T}\is F^*\mF_{T}\is \mF_{T}$ 
and we denote by $\mF_{T,w}\in D(T)^{F_{w}}$ the corresponding object.
We denote by 
\[\chi_\mF=\on{Tr}(F,\pi_G^*(\mF))\in C(G^F)\]
the  class function assocated to $\mF$
under the Sheaf-Function correspondence. Here $\pi_G:G\to\frac{G}{G}$
is the natural quotient map.

\begin{proposition}\label{characterization of stable functions}
(1) For any $F$-equivaraint stable complex $\mF\in D^{st}_G(G)^F$ the correspoinding class function 
$\chi_\mF\in C^{st}(G^F)$ is stable and the corresponding $\gamma$-function is given by 
\[\gamma_{\chi_\mF}([\theta])=\on{Tr}(F_w, H^*_c(T,\mF_{T,w}\langle-\dim G/B\rangle\otimes\mL_{\theta}^{-1}))=
q^{\dim G/B}\on{Tr}(F_w, H^*_c(T,\mF_{T,w}\otimes\mL_{\theta}^{-1}))\]
here $\mL_\theta\in D(T)^{F_w}$
is the Kummer local system assoicated to $\theta$.

(2)The
assignment $\mF\to \chi_\mF$ defines a 
    surjective map 
    \[\chi:D^{st}_G(G)^F\to C^{st}(G^F)\]
  
\end{proposition}
    \begin{proof}
    This is \cite[Proposition 6.7]{C3}
    \end{proof}

\subsection{Formula for transfer for stable functions}

Let $\hat\rho:\hat G'\to \hat G$ be a morphism as in 
Section \ref{construction}.
Assume $G',T'$ are also defined over $\mathbb F_q$
with geometric Frobenious $F:G'\to G'$
and $\hat\rho$ is  
compatible with the endomorphism $\hat F$ of the dual groups $\hat G, \hat G'$.
Then $\hat\rho$ induces a map between the $\hat F$-stable 
semisimple conjugacy classes 
\[\hat\rho_{ss}: (\hat T'//\rW')^{\hat F}\to (\hat T//\rW)^{\hat F}\]
and we define the transfer map b
\begin{equation}
\on{t}_{\hat\rho}: C^{st}(G^F)\longrightarrow
C^{st}((G')^{F})
\end{equation}
to be the unique ring homomorohism fitting into the following commutative diagram
\[\xymatrix{C^{st}(G^F)\ar[r]^{\on{t}_{\hat\rho}}\ar[d]_{\simeq}^{\eqref{description}}&C^{st}((G')^F)\ar[d]_{\simeq}^{\eqref{description}}\\
\overline{\mathbb Q}_\ell[(\hat T//\rW)^{\hat F}]\ar[r]^{\hat\rho_{ss}^*}&\overline{\mathbb Q}_\ell[(\hat T//\rW)^{\hat F}]}\]

We have the following characterization of  
transfer maps:
\begin{lemma}\label{char of transfer}
    For any $f\in C^{st}(G^F)$ the transfer $f':=\on{t}_{\hat\rho}(f)\in C^{st}((G')^F)$ is the unique stable function characterized by the following property: For any 
$w'\in\rW'$ and any character 
    $\theta':(T')^{F_{w'}}\to\overline{\mathbb Q}_\ell$, let $w\in\rW$ be a lift in~\eqref{lifting} and let
    $\theta=\rho_{T,T'}^*\theta':T^{F_w}\to (T')^{F_{w'}}\to \overline{\mathbb Q}_\ell$ be the pull-back of 
    $\theta'$ along the map 
    $\rho_{T,T'}:T^{F_w}\to (T')^{F_{w'}}$. 
    We have 
    \[\gamma_{f'}([\theta'])=\gamma_f([\theta]).\]
    
\end{lemma}
\begin{proof}
    It follows from the constuction.
\end{proof}
The following theorem shows that the transfer morphisms between stable complexes can be viewed as a geometrization of the transfer maps between stable fucntions:
Consider the functor
\[\on{T}_\rho':=\on{T}_{\hat\rho}\langle\dim G'/B'-\dim G/B\rangle:D^{st}_G(G)\to D^{st}_{G'}(G')\]

\begin{thm}\label{geometrization}
We have the following commuative diagram
\[\xymatrix{D^{st}_G(G)^F\ar[r]^{\on{T}'_\rho}\ar[d]^\chi&D^{st}_{G'}(G')^F\ar[d]^\chi\\
C^{st}(G^F)\ar[r]^{\on{t}_{\hat\rho}}&C^{st}((G')^F)}\]
    
\end{thm}
\begin{proof}
Let  $\mF\in D_G^{st}(G)$ be a stable complex. Denote by $\mF'=\on{T}'_{\hat\rho}(\mF)$, $f=\chi_\mF$,
and $f'=\chi_{\mF'}$.
Let $w,w',\theta,\theta'$ be as in Lemma \ref{char of transfer}. 
We have
\[\on{Res}_{\rW'}(\mF')\is
\on{Res}_{\rW'}\circ\on{Ind}_{\rW'}\circ(\rho_{T,T'})_!\on{Res}_\rW(\mF)\langle\dim G'/B'-\dim G/B\rangle\stackrel{~\eqref{res-ind}}\is\]\[\is(\rho_{T,T'})_!\on{Res}_\rW(\mF)\langle\dim G'/B'-\dim G/B\rangle\]
and the projection formula implies an isomorphism
\[H_c^*(T',\on{Res}_{\rW'}(\mF')\langle-\dim G'/B'\rangle\otimes\mL_{\theta'}^{-1})\is H_c^*(T',(\rho_{T,T'})_!\on{Res}_\rW(\mF)\langle-\dim G/B\rangle\otimes\mL_{\theta'}^{-1})\is\]
\[\is H_c^*(T,\on{Res}_\rW(\mF)\langle-\dim G/B\rangle\otimes\mL_{\theta}^{-1}) \]
intertwines the endomorphisms $F_{w'}$ and $F_w$. 
Applying Proposition \ref{characterization of stable functions}, we obtain
\[\gamma_{f'}([\theta'])=q^{\dim G'/B'}\on{Tr}(F_{w'},H_c^*(T',\on{Res}_{\rW'}(\mF')\otimes\mL_{\theta'}^{-1}))=q^{\dim G/B}\on{Tr}(F_{w},H_c^*(T,\on{Res}_{\rW}(\mF)\otimes\mL_{\theta}^{-1}))=\]
\[=\gamma_{f}([\theta])\]
and the theorem follows from the characterization of transfer maps in Lemma \ref{char of transfer}.
\end{proof}

As an application, we obtain the following formula 
of transfer maps in terms of Deligne-Lusztig inductions.
Let $f\in C^{st}(G^F)$ be a stable function.
For any $w\in\rW$ let 
$f_w\in C(T^{F_w})$ be the unique function 
such that for any character $\theta$ on $T^{F_w}$ we have
\begin{equation}\label{f_w}
    f_w\star\theta=\gamma_f([\theta])\theta
\end{equation}
For any $w'\in\rW'$, there exists an embedding 
$\iota_{w'}:T'\to G'$ intertwines the $w'$-twisted Frobenius $F_{w'}$ on $T'$ with the Frobenius $F$ on $G'$.
Denote by $T_{w'}'=\iota_{w'}(T')\subset G' $ which is a $F$-stable maximal torus of $G'
$. We have an isomorphism 
$\iota_{w'}:(T')^{F_{w'}}\is (T'_{w'})^F$.
Recall 
 the Deligne-Lusztig induction map (see \cite[Section 3.3]{LL1})
\[\on{R}_{T'_{w'}}^{G'}:C((T')^{F_{w'}})\is C((T_{w'}')^F)\to C((G')^F).\]
For any $w\in\rW'$ we pick a lift $w\in\rW$
as in diagram~\eqref{lifting} and denote by 
\[f_{w'}=(\rho_{T,T'})_!f_w\in C((T')^{F_{w'}})\]
where $\rho_{T,T'}:T\to T'$ is the dual morphism in~\eqref{lifting}
\begin{corollary}\label{LL formula}
Let $f\in C^{st}(G^F)$ be a stable function.
The transfer $f'=\on{t}_{\hat\rho}(f)\in C^{st}((G')^F)$  is given by the following formula
\begin{equation}\label{fromula}
    f'=\frac{1}{|W'|}\sum_{w'\in\rW'}(-1)^{l(w')}q^{-\dim G'/B'}\on{R}_{T'_{w'}}^{G'}(f_{w'})
\end{equation}

\end{corollary}
\begin{proof}
Let $\mF\in D^{st}_G(G)$ be a $F$-equivaraint stable complex on $G$. Let 
$\mF_{T}=\on{Res}_\rW(\mF)\in D_{\rW}(T)^F$ and $\mF_{T'}=(\rho_{T,T'})_!\mF_{T}\in D_{\rW'}(T')^F$.
We have $\on{T}'_{\hat\rho}(\mF)=\on{Ind}_{\rW'}(\mF_{T'})\langle\dim G'/B'-\dim G/B\rangle$ and, by Theorem \ref{geometrization}, it suffices to show that 
\[f'=\on{Tr}(F,\pi_{G'}^*\on{T}'_{\hat\rho}(\mF))=
q^{\dim G/B-\dim G'/B'}
\on{Tr}(F,\pi_{G'}^*\on{Ind}_{\rW'}(\mF_{T'}))
\]
is equal to the right hand side of~
\eqref{fromula}. 
For any $w'\in\rW'$,
we have an isomorphism $F_{w'}^*\mF_{T'} \is F^*(w')^*\mF_{T'}\is F^*\mF_{T'}\is \mF_{T'}$ 
and we denote by $\mF_{T',w'}\in D(T')^{F_{w'}}$ the corresponding object. 
By Proposition \ref{characterization of stable functions}, we have
\[\on{Tr}(F_{w'},\mF_{T',w'})=q^{-\dim G/B}f_{w'}.\]
On the other hand, 
\cite[Theorem 5.2.6]{LL3} implies 
\[\on{Tr}(F,\pi_{G'}^*\on{Ind}_{\rW'}(\mF_{T'}))=\frac{1}{|\rW'|}\sum_{w'\in\rW'}(-1)^{l(w')}\on{R}_{T'_{w'}}^{G'}(\on{Tr}(F_{w'},\mF_{T',w'}))\]
here the factor $(-1)^{l(w')}$
is due to the nomalization of $\rW'$-action in Remark \ref{sign}.
Combining the about two equations we arrive the desired formula
\[f'=q^{\dim G/B-\dim G'/B'}\on{Tr}(F,\pi_{G'}^*\on{Ind}_{\rW'}(\mF_{T'}))=\frac{1}{|W'|}\sum_{w'\in\rW'}(-1)^{l(w')}q^{-\dim G'/B'}\on{R}_{T'_{w'}}^{G'}(f_{w'})\]
\end{proof}

\begin{example}
Consider the case $\hat\rho=\on{id}:\hat G\to \hat G$ is the identify map and 
$f=\delta$ is  the delta function at the unit element. We have $\delta=\on{t}_{\hat\rho}(\delta)
$ and $f_w=\delta_w$ is the delta fucntion at the unit of $T_w^F$ and 
the formula above says that
 $\delta=\frac{1}{|W|}\sum_{w\in\rW}(-1)^{l(w)}q^{-\dim G/B}\on{R}_{T_{w}}^{G}(\delta_w)$.
 This agrees with the formula in \cite[Proposition 12.20]{DM}.

\end{example}

\begin{remark}
Using \cite[Lemma 4.2.1]{LL1},
    one can check that the formula above agrees with  
    Theorem 1.0.1 in  \cite{LL3}.
\end{remark}


\begin{thebibliography}{ABPS16}

\bibitem[B]{B}
A. Borel,  Automorphic L-functions, Automorphic forms, representations and
L-functions Proc. Sympos. Pure Math., Oregon State Univ., Corvallis, Ore., 1977), Part
2. Amer. Math. Soc., Providence, R.I., 1979.

\bibitem[BK]{BK}
A. Braverman. D. Kazhdan, 
$\gamma$-sheaves on reductive groups,
, Studies in Memory of Issai Schur,
Chevaleret-Rehovot, 2000, Birkhuser Boston, Boston, MA, 2003, 27-47.
\bibitem[BT]{BT}
R. Bezrukavnikov, T. Deshpande, 
Vanishing sheaves and the geoemtric Whittaker model, arXiv:2310.14834.

\bibitem[C1]{C1}
T.-H. Chen, A vanishing conjecutre: $GL(n)$ case, Selecta Math. 28, 13 (2022)

\bibitem[C2]{C2}
T.-H. Chen, 
On a conjecture of Braverman-Kazhdan,  J. Amer. Math. Soc. 35 (2022), 1171-1214.

\bibitem[C3]{C3}
T.-H. Chen, 
Toward the depth zero stable Bernstein center conjecture, arXiv:2303.13454.





\bibitem[DL]{DL}
P. Deligne, G. Lusztig,
Representations of reductive groups over finite fields, Annals of Math. 103
(1976), 103-161.


\bibitem[DM]{DM}
F. Digne, J. Michel,
Representations of Finite Groups of Lie Types,
London Mathematical Society Student Texts 21.


\bibitem[G]{G} 
V. Ginzburg, Nil Hecke algebra and Whittaker $D$-modules,
Lie Groups, Geometry, and Representation Theory. Progress in Mathematics, vol 326. Birkhäuser.


\bibitem[GL]{GL}
O. Gabber, F. Loeser,
Faisceaux pervers $\ell$-adiques sur un tore, Duke Math. J. 83 (1996), no. 3, 501-606.

\bibitem[K]{K}
R. Kottwitz, 
Stable trace formula: Cuspidal tempered terms, 
Duke Mathematical Journal, Duke Math. J. 51(3), 611-650, (September 1984).

\bibitem[L]{L} 
G. Lonergan, A Fourier transform for the quantum Toda lattice,
Selecta Mathematica, November 2018, Volume 24, Issue 5, 4577-4615.

\bibitem[LL1]{LL1}
G. Laumon, E. Letellier, 
Notes on a conjecture of Braverman-Kazhdan,
arXiv:1906.07476

\bibitem[LL2]{LL2}
G. Laumon, E. Letellier, 
A derived Lusztig equivalence, 
Forum of Mathematics, Sigma , Volume 11 , 2023.

\bibitem[LL3]{LL3}
G. Laumon, E. Letellier, 
Notes on a conjecture of Braverman-Kazhdan,
arXiv version 1, arXiv:1906.07476v1




\bibitem[S]{S}
T. Springer,  Reductive groups, Automorphic forms, representations and
L-functions Proc. Sympos. Pure Math., Oregon State Univ., Corvallis, Ore., 1977), Part
1. Amer. Math. Soc., Providence, R.I., 1979.




\bibliographystyle{plain}



\end{thebibliography}
\end{document}